\documentclass[11pt,a4paper]{amsart}
\usepackage{fullpage}
\usepackage{bbm,color}
\usepackage{amssymb,amscd,hyperref}
\usepackage[all]{xy}
\newtheorem{thm}{Theorem}[section]
\newtheorem{lem}[thm]{Lemma}
\newtheorem{prop}[thm]{Proposition}
\newtheorem{cor}[thm]{Corollary}
\theoremstyle{definition}
\newtheorem{rem}[thm]{Remark}
\newtheorem{defn}[thm]{Definition}

\def\Z{\mathbb{Z}}
\def\Q{\mathbb{Q}}

\def\HR{H\! R}
\def\ev0{\mathit{Ev}_0}

\def\leq{\leqslant}
\def\geq{\geqslant}

\def\ra{\rightarrow}

\def\Sp{\text{$\sf{Sp}$}}
\def\Sym{\text{$\sf{Sym}$}}

\def\smod{\text{$s\sf{mod}$}}

\def\sab{\text{$\sf{sAb}$}}
\def\chab{\text{$\sf{ch}$}}
\def\Chab{\text{$\sf{Ch}$}}
\def\sigmadg{\text{$\ch^\Sigma$}}
\def\sigmas{\text{$\sab^\Sigma$}}

\def\comsigmadg{C(\sigmadg)}
\def\comsigmas{C(\sigmas)}

\DeclareMathOperator{\colim}{colim}
\def\ie{\emph{i.e.}}

\def\id{\mathrm{id}}

\newcommand{\spec}{\mathrm{Sp}^{\Sigma}}

\newcommand{\Zt}{\widetilde{\mathbb{Z}}}
\newcommand{\CH}{\Chab}
\newcommand{\ch}{\chab}

\newcommand{\sm }{\wedge}

\newcommand{\iso}{\cong}
\newcommand{\sset}{\mathcal{S}}
\newcommand{\boxprod}{\Box}

\newcommand{\bS}{\mathbb{S}}

\newcommand{\cO}{\mathcal{O}}

 \long\def\bltext#1{{\color{blue}#1}}

\long\def\bltext#1{{#1}}   

\begin{document}
\title{
An algebraic model for commutative $H\Z$-algebras}
\author{Birgit Richter}
\address{Fachbereich Mathematik der Universit\"at Hamburg,
Bundesstra{\ss}e 55, 20146 Hamburg, Germany}
\email{richter@math.uni-hamburg.de}
\urladdr{http://www.math.uni-hamburg.de/home/richter/}

\author{Brooke Shipley}
\address{Department of Mathematics, Statistics, and Computer Science,
University of Illinois at Chicago, 508 SEO M/C, 851 S. Morgan Street,
Chicago, IL, 60607-7045, USA}
\email{bshipley@math.uic.edu}
\urladdr{http://homepages.math.uic.edu/~bshipley/}

\date{\today}

\begin{abstract}
We show that the homotopy category of commutative algebra spectra over
the Eilenberg-Mac\,Lane spectrum of {an arbitrary commutative
  ring $R$} 
is equivalent to the
homotopy category of $E_\infty$-monoids in unbounded chain
complexes {over $R$}. We do this by establishing a chain of
Quillen equivalences 
between the corresponding model categories. We also provide a Quillen
equivalence to commutative monoids in the category of functors
from the category of finite sets and injections to unbounded chain
complexes.
\end{abstract}
\maketitle
\section{Introduction}
{Let $R$ be an arbitrary commutative ring.} In \cite{s-alg} it
was shown that the model category of algebra 
spectra over the 
Eilenberg-Mac\,Lane spectrum, 
{$\HR$},  is
connected to the
model category of differential graded {$R$-algebras} 
via a chain of Quillen
equivalences. In this paper we extend this result to the case of
commutative 
{$\HR$-algebra } spectra. 
As a guiding example we consider the function spectrum from a space
$X$ to the Eilenberg-Mac\,Lane spectrum of a commutative ring $R$,
$F(X,\HR)$. As $R$ is commutative, $F(X,\HR)$ is a commutative
$\HR$-algebra spectrum whose homotopy groups are the cohomology groups
of the space $X$ with coefficients in $R$:
$$ \pi_{-n}F(X,\HR) \cong H^n(X;R).$$
The singular cochains on $X$ with coefficients in $R$, $S^*(X;R)$, give
a chain model of the cohomology of $X$ by regrading. We set
$$ C_{-*}(X;R) := S^*(X;R).$$
{Note that for $R=\mathbb{F}_p$ the Steenrod operations on
  $H^*(X;R)$ can be constructed from the $\cup_i$-products. These are
  chain homotopies that measure the failure of the cup-product to
  produce a strictly graded commutative product of cochains. Thus, in general,
  one cannot expect to find a model of the singular cochains of a
  space that is a differential graded commutative $R$-algebra. Instead, one
  must work with $E_\infty$-algebra structures. \bltext{See
    also~\cite[Theorem 2]{cenkl}.} A
  notable exception are rational cochains of a space with the Sullivan
cochains as a strictly differential graded commutative model. }

We establish a chain of Quillen equivalences between commutative
{$\HR$}
-algebra
spectra, 
{$C(\HR\text{-mod})$}, and differential graded
{$E_\infty$-$R$-algebras}, 
\bltext{$E_\infty\Chab_R$}:
{$$ \xymatrix{
{C(\HR\text{-mod})} \ar@<0.5ex>[r]^Z  &
\ar@<0.5ex>[l]^U {C(\Sp^\Sigma(\bltext{\smod_R}))}
\ar@<-0.5ex>[r]_{\Phi^*N}   &
\ar@<-0.5ex>[l]_{L_N} {C(\Sp^\Sigma(\chab_R))}  \ar@<0.5ex>[r]^i  &
\ar@<0.5ex>[l]^{C_0} {C(\Sp^\Sigma(\Chab_R))}\ar@<0.5ex>[d]^{R_\varepsilon}
\\
 & &
{E_\infty\Chab_R}  \ar@<0.5ex>[r]^>>>>>>>{F_0}
&{E_\infty(\Sp^\Sigma(\Chab_R))} \ar@<0.5ex>[l]^>>>>>>{\ev0}
\ar@<0.5ex>[u]^{L_\varepsilon}
}$$}
Here, our intermediary categories include symmetric spectra
($\Sp^\Sigma$) over the categories of 
{simplicial $R$-modules,\bltext{ ($\smod_R$), } 
non-negatively graded chain complexes \bltext{over $R$}, \bltext{($\ch_R$)}, and
unbounded  chain complexes over $R$, \bltext{($\CH_R$)}.   The functors
  will be introduced in the sections below. 

The fact that there is such a Quillen equivalence should not be surprising, but 
to our knowledge, this result cannot be found in the literature.  

In the context of infinite loop space theory, $E_\infty$-ring spectra,
and their units, the theory of $\mathcal{I}$-spaces is important; see
\cite{sasch}. Here
$\mathcal{I}$ is the category of finite sets and injections and
$\mathcal{I}$-spaces are functors from $\mathcal{I}$ to simplicial
sets. More generally, functor categories from $\mathcal{I}$ to
categories of modules feature as $\mathit{FI}$-modules in the work of
Church, Ellenberg, Farb \cite{CEF} and others. We relate symmetric
spectra in unbounded chain complexes over $R$} via a chain of Quillen
equivalences to the category of unbounded $\mathcal{I}$-chain
complexes and prove that commutative monoids in this category,
$C(\Chab_R^\mathcal{I})$,  provide
an alternative model for commutative $HR$-algebra spectra.
In fact,
there is a chain of Quillen equivalences between 
{$C(\HR\text{-mod})$} and $E_\infty$-monoids in unbounded
  $\mathcal{I}$-chain complexes {over $R$},
  $E_\infty(\CH_R^\mathcal{I})$, that passes via 
$E_\infty(\Sp^\Sigma(\Chab_R))$ and $E_\infty\Chab_R$. The
rigidification result of Pavlov and Scholbach \cite[3.4.4]{ps2} for symmetric
spectra implies that the model category $E_\infty(\Chab_R^\mathcal{I})$ is
Quillen equivalent to the one of commutative monoids in
$\Chab_R^\mathcal{I}$, that is $C(\Chab_R^\mathcal{I})$. Taking these results
together we obtain a chain of Quillen equivalences between commutative
{$\HR$-algebra spectra and commutative monoids in
  $\mathcal{I}$-chain complexes over $R$.}
See Theorem~\ref{comm.hz.I.ch}. We expect that our comparison result 
makes it possible to find explicit commutative $\mathcal{I}$-chain
models for certain commutative 
{$\HR$-algebras}
and there is ongoing work on this by Richter, Sagave
and Schulz with applications to logarithmic structures on commutative
ring spectra in mind. 

If \bltext{$R= \Q$} is the field of 
rational numbers we can prolong our
chain of Quillen equivalences and obtain a comparison (Corollary
\ref{cor:rat}) between
commutative $H\Q$-algebra spectra and differential graded commutative
$\Q$-algebras.

Mike Mandell showed in \cite[7.11]{man03} that for every commutative
ring $R$ the homotopy categories of $E_\infty$-$\HR$-algebra spectra
and of $E_\infty$ monoids in the category of unbounded $R$-chain
complexes are equivalent. He also claims in \emph{loc.\,cit.}~that he
can improve this equivalence of homotopy categories to an actual chain
of Quillen equivalences. He suggests using the methods of
\cite{schsh}, but only associative monoids are treated there. 

Our approach is different from Mandell's
because we work in the setting of symmetric spectra. The idea to
integrate the symmetric groups into the monoidal structure to
construct a symmetric monoidal category of spectra is due to
Jeff Smith. Our arguments heavily rely on combinatorial and monoidal
features of the category of symmetric spectra in the categories of
simplicial sets, simplicial {$R$-modules,} 
non-negatively graded chain complexes ($\ch_R$) and unbounded chain
complexes ($\CH_R$).

The structure of the paper is as follows: We recall some basic facts
and some model categorical features of symmetric spectra in section
\ref{sec:background}. In section \ref{sec:cm} we recall results from
Pavlov and Scholbach \cite{ps1,ps2} that establish model
structures on commutative ring spectra in the cases that arise as
intermediate steps in our chain of Quillen equivalences and we also
recall their rigidification result. We sketch how to use
methods from Chadwick-Mandell \cite{cm} for an alternative proof. The Quillen
equivalence between commutative 
{$\HR$}-algebra spectra and commutative
symmetric ring spectra in simplicial 
{$R$-modules} can be found in
section \ref{sec:hz} as Theorem \ref{thm:hzvssab}. The Quillen
equivalence between the latter model category and commutative
symmetric ring spectra in non-negatively graded chain complexes is based on
the Dold-Kan correspondence and is stated as Theorem
\ref{thm:qecommspectrasabchab} in section \ref{sec:dk}. There is a
natural inclusion functor $i \colon \chab \ra \CH$ and the Quillen
equivalence between commutative symmetric ring spectra in $\chab$ and
in $\CH$ (see corollary \ref{cor:chCH}) is based on this functor. In
section \ref{sec:einfty} we establish a Quillen equivalence between
$E_\infty$-monoids in symmetric spectra in unbounded chain complexes
and $E_\infty$-monoids in unbounded 
{chain complexes.} The link with
$E_\infty$-monoids and commutative monoids in the diagram category
of chain complexes  {indexed}
by the category of finite sets and injections is worked out in section
\ref{sec:ichains}.

\textbf{Acknowledgement:} This material is based upon work supported by
the National Science Foundation  under Grant No. 0932078000 while the
authors were in
residence at the Mathematical Sciences Research Institute in Berkeley
California, during the Spring 2014 program on Algebraic Topology.
The second author was also supported during this project by the NSF
under Grants No. 1104396 and 1406468.  We
are grateful to Dmitri Pavlov and Jakob Scholbach for sharing draft
versions of \cite{ps1,ps2} with us. We thank Benjamin Antieau and
Steffen Sagave for helpful comments on an earlier version of this paper.

\section{Background} \label{sec:background}
In the following we will consider model category structures that are
transferred by an adjunction. Given an adjunction
$$ \xymatrix@1{ \mathcal{C} \ar@<0.5ex>[r]^{L} &   \ar@<0.5ex>[l]^{R}
\mathcal{D}} $$
where $\mathcal{C}$ is a model category and $\mathcal{D}$ is a
bicomplete category,
we call a model structure on $\mathcal{D}$ \emph{right-induced} if the weak
equivalences and fibrations in $\mathcal{D}$ are determined by the right
adjoint functor $R$.

We use the general setting of symmetric spectra as in \cite{H}. Let
$(\mathcal{C}, \otimes, \mathbf{1})$ be a bicomplete closed symmetric
monoidal category
and let $K$ be an object of $\mathcal{C}$. A symmetric sequence in
$\mathcal{C}$ is a family of objects $X(n) \in \mathcal{C}$ with $n
\in \mathbb{N}_0$ such that the $n$th level $X(n)$ carries an action
of the symmetric group $\Sigma_n$. Symmetric sequences form a category
$\mathcal{C}^\Sigma$ whose morphisms are given by families of
$\Sigma_n$-equivariant morphisms $f(n)$, $n\geq 0$ . For every $r \geq 0$
there is a functor
$G_r\colon \mathcal{C} \ra \mathcal{C}^\Sigma$ with
$$G_r(C)(n) = \begin{cases} \Sigma_n \times C, & \text{ for } n=r, \\
\varnothing, & n \neq r, \end{cases}$$
where $\varnothing$ denotes the initial object of $\mathcal{C}$.
 Here $\Sigma_n
\times C = \bigsqcup_{\Sigma_n} C$ carries the $\Sigma_n$-action that
permutes the summands.

We consider the symmetric
sequence $\Sym(K)$ whose $n$th level is $K^{\otimes n}$. {Here
  we follow the usual convention that $K^{\otimes 0}$ is the unit
  $\mathbf{1}$.}  
The category $\mathcal{C}^\Sigma$ inherits a symmetric monoidal
structure from $\mathcal{C}$: for $X,Y \in \mathcal{C}^\Sigma$ we set
$$ (X \odot Y)(n) = \bigsqcup_{p+q=n} \Sigma_n \times_{\Sigma_p \times
\Sigma_q} X(p) \otimes Y(q).$$
{It is straightforward to show (see for instance \cite[\S
  7]{H})  that $\Sym(K)$ is a
  commutative monoid in $\mathcal{C}^\Sigma$.}

The category of \emph{symmetric spectra} (in $\mathcal{C}$ with respect to $K$),
$\Sp^\Sigma(\mathcal{C}, K)$, is the category of right $\Sym(K)$-modules in
$\mathcal{C}^\Sigma$. Explicitly, a symmetric spectrum is a family of
$\Sigma_n$-objects $X(n) \in \mathcal{C}$ together with
$\Sigma_n$-equivariant maps
$$ X(n) \otimes K \ra X(n+1)$$
for all $n \geq 0$ such that the composites
$$ X(n) \otimes K^{\otimes p} \ra X(n+1) \otimes K^{\otimes p-1} \ra
\ldots \ra X(n+p)$$
are $\Sigma_n \times \Sigma_p$-equivariant for all $n,p\geq 0$. Morphisms in
$\Sp^\Sigma(\mathcal{C}, K)$ are morphisms of symmetric sequences that
are compatible with the right $\Sym(K)$-module structure.

There is an evaluation functor $\mathit{Ev}_n$ that maps an $X \in
\Sp^\Sigma(\mathcal{C}, K)$ to $X(n) \in \mathcal{C}$. This functor
has a left adjoint,
$$F_n \colon \mathcal{C} \ra \Sp^\Sigma(\mathcal{C}, K)$$
such that $F_n(C)(m)$ is the initial object for $m < n$ and
$$ F_n(C)(m) \cong \Sigma_m \times_{\Sigma_{m-n}} C \otimes K^{\otimes
  m-n}, \text{ if } m \geq n.$$
Note that $F_n(C) \cong G_n(C) \odot \Sym(K)$.

Symmetric spectra form a symmetric monoidal category
$(\Sp^\Sigma(\mathcal{C},K), \wedge, \Sym(K))$ such that for $X,Y \in
\Sp^\Sigma(\mathcal{C},K)$
$$ X\wedge Y = X \odot_{\Sym(K)} Y.$$
Here {$X \odot_{\Sym(K)} Y$ denotes the coequalizer of the
  diagram $\xymatrix@1{X \odot {\Sym(K)} \odot Y \ar@<-0.5ex>[r]
    \ar@<0.5ex>[r] & X \odot  Y}$ where we use the right action of
  $\Sym(K)$ on $X$ and} we use the right action of $\Sym(K)$ on $Y$
after applying the 
twist-map in the symmetric monoidal structure on
$\mathcal{C}^\Sigma$.

A crucial map is
\begin{equation} \label{eq:lambda}
\lambda\colon F_1K \ra F_0\mathbf{1};
\end{equation}
it is given as the adjoint to the identity map $K \ra
\mathit{Ev}_1F_0\mathbf{1} = K$.

We recall the basics about model category structures on symmetric
spectra from \cite{H}: If $\mathcal{C}$ is a closed symmetric monoidal
model category which is left proper and cellular and if $K$ is a
cofibrant object of $\mathcal{C}$, then there is a
\emph{projective model structure on the category
$\Sp^\Sigma(\mathcal{C},K)$} \cite[8.2]{H},
$\Sp^\Sigma(\mathcal{C},K)_{\text{proj}}$,
 such that the fibrations and weak equivalences are levelwise
fibrations and weak equivalences in $\mathcal{C}$ and such that the
cofibrations are determined by the left lifting property with respect
to the class of acyclic fibrations.

This model structure has a Bousfield localization with respect to the
set of maps
$$\{ \zeta^{QC}_n\colon F_{n+1}(QC \otimes K) \ra F_n(QC), n \geq 0\}$$
where $Q(-)$ is a cofibrant replacement and $C$ runs through the
domains and codomains of the generating cofibrations of
$\mathcal{C}$. The map $\zeta^{QC}_n$ is adjoint to the inclusion map
into the component of $F_n(QC)(n+1)$ corresponding to the identity
permutation. We call the Bousfield localization of
$\Sp^\Sigma(\mathcal{C},K)_{\text{proj}}$ at this set of maps the
\emph{stable model structure on $\Sp^\Sigma(\mathcal{C},K)$} and
denote it by $\Sp^\Sigma(\mathcal{C},K)^s$. 

As we are interested in commutative monoids in symmetric spectra, we
use positive variants of the above mentioned model structures: Let
$\Sp^\Sigma(\mathcal{C},K)^+_{\text{proj}}$ be the model structure
where fibrations are maps that are fibrations in each level $n \geq 1$
and weak equivalences are levelwise weak equivalences for positive
levels. The cofibrations are again determined by their lifting
property and they turn out to be isomorphisms in level zero (compare \cite[\S
14]{mmss}).  By adapting the localizing set and considering only
positive $n$, we get the positive stable model structure on
$\Sp^\Sigma(\mathcal{C},K)$ and we denote it by
$\Sp^\Sigma(\mathcal{C},K)^{s,+}$. 

{
\begin{rem}
We consider several examples of categories $\mathcal{C}$ with
different choices of objects $K \in \mathcal{C}$. Despite the name,
the stable model structure on $\Sp^\Sigma(\mathcal{C},K)$ does not
have to define a stable model category in the sense that the category
is pointed with a homotopy category that carries an invertible
suspension functor. Proposition \ref{prop:ci} for instance makes this
explicit \bltext{in the case when $K$ is the unit of the symmetric monodical structure on $\mathcal{C}$.}
\end{rem}
}

\section{Model structures on algebras over an operad
over {$\Sp^\Sigma(\mathcal{C})$ for $\mathcal{C} = \chab, \sab,
\CH$}} \label{sec:cm}

{From now on we restrict to the case $R=\Z$ in order to ease
  notation. The proofs work in general. }

Establishing right-induced model structures for commutative monoids in model
categories is hard. Sometimes it is not possible, for instance there
is no right-induced model structure on differential graded commutative 
rings, because the free functor does not respect acyclicity.
However, if the underlying model category is nice enough, then such
model structures can be established. In broader generality, one might
ask whether algebras over operads possess a right-induced model
structure. In our setting we will apply the results of Pavlov and
Scholbach. They show in 
{\cite[5.10]{ps1}} 
and \cite[3.4.1]{ps2} that
for a tractable, pretty small, left proper, h-monoidal, flat symmetric
monoidal model category $\mathcal{C}$ the category of $\mathcal{O}$-algebras in
$\Sp^\Sigma(\mathcal{C},K)^{s,+}$ has a right-induced model
structure. Here $\mathcal{O}$ is an operad {in
  $\mathcal{C}$}. See {\em loc.~cit.} for an 
explanation of the assumptions. These conditions are satisfied for the
model categories of
simplicial abelian groups and both non-negatively graded and unbounded chain
complexes. {Hence, using their results, we obtain: 
\begin{thm} \label{thm:existence-ps}
The category of $\mathcal{O}$-algebras in
$\Sp^\Sigma(\mathcal{C},K)^{s,+}$ has a right-induced model
structure for $\mathcal{C} = \sab, \chab, \CH$, any $K$ and any operad
$\cO$ in $\mathcal{C}$.
\end{thm}}

{We follow the convention that an $E_\infty$-operad
  $\mathcal{P}$ in $\CH$ (or $\chab, \sab$) is a symmetric unital operad whose
  augmentation induces a weak equivalence to the operad that describes
commutative monoids. For sake of brevity we call algebras over an
$E_\infty$-operad \emph{$E_\infty$-monoids}.}  
Pavlov-Scholbach also prove a rigidification
theorem  
{\cite[7.5]{ps1}},  
\cite[3.4.4]{ps2}). We apply this to the
case of $E_\infty$-monoids and in this case it provides a Quillen
equivalence between the model category of $E_\infty$-monoids in
$\Sp^\Sigma(\mathcal{C},K)^{s,+}$ and commutative monoids in
$\Sp^\Sigma(\mathcal{C},K)^{s,+}$. {Related rectification
  results in the setting 
of spaces instead of chain complexes are due to \cite{gh} and
\cite{sasch}. Berger and Moerdijk obtain general results about
rectifications of homotopy algebra structures in \cite{bm2}.}

Other approaches to model structures for commutative monoids in
symmetric spectra and rigidification results can be found for instance
in work by
John Harper \cite{Har}, David White \cite{W}, and  Steven Chadwick and
Michael Mandell \cite{cm}.


In the following we sketch {an alternative proof of the
  existence of a positive stable right-induced model structure for the
category of symmetric spectra in the 
category of unbounded chain complexes, $\Sp^\Sigma(\CH, \Z[1])$, where
$\Z[1]$ denotes the chain
complex which is concentrated in chain degree one with chain group
$\Z$. This proof uses  a modification of} 
the methods used  by Chadwick-Mandell \cite{cm}.  
 A similar proof works for the categories  of
symmetric spectra in simplicial abelian groups, $\Sp^\Sigma(\sab,
\tilde{\Z}({\mathbb{S}}^1))$, with $K=
\tilde{\Z}({\mathbb{S}}^1)$ the reduced free abelian simplicial group generated
by the simplicial $1$-sphere, and for symmetric spectra in the
category of non-negatively graded chain
complexes,  $\Sp^\Sigma(\ch, \Z[1])$. 

{A reader who is just interested in the application of these
  results is invited to resume reading in section \ref{sec:hz}.} 

\begin{thm}\label{thm.op.models}
Let $\cO$ be an operad {in $\CH$}.  Then the category
$\cO(\Sp^\Sigma(\CH))$ of 
$\cO$-algebras over $\Sp^\Sigma(\CH)$ is a model category with fibrations and
weak equivalences created in the positive stable model structure on
$\Sp^\Sigma(\CH)$.
 \end{thm}

\begin{thm}\label{thm.op.compare}
Let $\phi\colon \cO \to \cO'$ be a map of operads.  The induced adjoint functors
$$\xymatrix@1{ {\cO(\Sp^\Sigma(\CH))} \ar@<0.5ex>[r]^{L_{\phi}} &
  {\cO'(\Sp^\Sigma(\CH))} \ar@<0.5ex>[l]^{R_{\phi}}}$$
form a Quillen adjunction.  This is a Quillen equivalence if
$\phi(n)\colon \cO(n) \to \cO'(n)$ is a (non-equivariant) weak equivalence
for each $n$.

In particular, if $\varepsilon$ is the augmentation from any
$E_{\infty}$ operad to the commutative operad, 
then it induces a Quillen equivalence between the categories of
$E_{\infty}$ monoids and of commutative monoids in $\Sp^\Sigma(\CH)$. 
\end{thm}

The proofs of both of these theorems use the following statement,
which is a translation  of \cite[15.5]{mmss} to $\Sp^\Sigma(\CH)$ with a
slight generalization based on {\cite[Remark 8.3(i)]{cm}}. 
As a model for
$E\Sigma_n$ in the category $\Sp^\Sigma(\CH)$ we take $F_0$ applied to
the normalization of the free  simplicial abelian group generated by the
nerve of the translation category of the symmetric group $\Sigma_n$.

\begin{prop}\label{prop.gen.15.5} Let $X$ and $Z$ be objects in
  $\Sp^\Sigma(\CH)$.
\begin{enumerate}
\item
Let $K$ be a chain complex, assume $X$ has a $\Sigma_i$ action, and $n > 0$.
Then the quotient map $$q\colon E\Sigma_{i+} \sm_{\Sigma_i}(( F_nK)^{\sm i} \sm X)
\to (F_n K)^{\sm i} \sm X) / \Sigma_i$$
is a level homotopy equivalence.
\item
For any positive cofibrant object $X$ and any $\Sigma_i$-equivariant object $Z$,
$$ q\colon E\Sigma_{i+} \sm_{\Sigma_i} (Z \sm X^{\sm i}) \to (Z \sm
X^{\sm i}) / \Sigma_i$$ 
is a $\pi_*$-isomorphism.
\end{enumerate}
\end{prop}

\begin{proof}
First, the proof of \cite[15.5]{mmss}, easily translates to the
setting of $\Sp^\Sigma(\CH)$ from $\Sp^\Sigma(\sset)$ considered there.
The key point is that
if $q \geq ni$, then $E\Sigma_i \times \Sigma_q \to \Sigma_q$ is a
$(\Sigma_i \times \Sigma_{q-ni})$-equivariant homotopy equivalence.
As mentioned in {\cite[8.3(i)]{cm}}, 
the proof of the first statement in
\cite[15.5]{mmss} still works when $X$ has a $\Sigma_i$ action because
the $\Sigma_i$ action remains free on $\Sigma_q$  (or $\cO(q)$ in the
explicit case there.)  Similarly the second statement here follows by
the same cellular filtration of $X$ as in  \cite[15.5]{mmss}.
\end{proof}

The proofs of both of the theorems above also require the following
definition and statement of properties.

\begin{defn} A chain map $i\colon A \to B$ in $\CH$ is an {\em
    h-cofibration} if each homomorphism $i_n\colon  A_n \to B_n$ has a
  section (or splitting).   These are the cofibrations in a model
  structure on $\CH$;
  see~\cite[3.4]{christensen.hovey},~\cite[4.6.2]{schwanzl-vogt},
  or~\cite[18.3.1]{may-ponto}.  We say a map $i\colon X \to Y$ in
  $\Sp^\Sigma(\CH)$ is an {\em h-cofibration} if each level $i_n\colon X_n \to
  Y_n$ is an h-cofibration as a chain map.
\end{defn}

Below we refer to $\Sigma_n$-equivariant h-cofibrations.  These
are $\Sigma_n$-equivariant maps for which the underlying
non-equivariant map is an h-cofibration.  We use the following
properties of h-cofibrations below.

\begin{prop}\label{prop.h.list}
\begin{enumerate}
\item[]
\item The generating cofibrations and acyclic cofibrations, in $\CH$
are h-cofibrations.
\item Sequential colimits and pushouts preserve h-cofibrations.
\item If $f$ and $g$ are two h-cofibrations in $\CH$, then their
  pushout product $f \boxprod
  g$ is also an h-cofibration.
\item If $f$ is an h-cofibration in $\CH$, then $F_i f$ is an
  h-cofibration in $\Sp^\Sigma(\CH)$.
\item For any $\Sigma_n$-equivariant object $Z$,  subgroup $H$ of
  $\Sigma_n$,  $\Sigma_n$-equivariant h-cofibration $f$, and $i \geq
  n$, the map $Z \sm_{H} F_i (f)$ is an h-cofibration.
\end{enumerate}
\end{prop}

We write $\cO I$ and $\cO J$ for the sets of maps in $\cO(\Sp^\Sigma(\CH))$
obtained by applying the free $\cO$-algebra functor to the
generating cofibrations $I$ and
generating acyclic cofibrations $J$ from~\cite{s-alg}.
Since $\Sp^\Sigma(\CH)$ is a combinatorial model category and the free
functor $\cO$ commutes with filtered direct limits, to prove
Theorem~\ref{thm.op.models} it is enough to prove the following lemma
by~\cite[2.3]{ss1}.

\begin{lem}\label{lem.operad.pout}
Every sequential composition of pushouts  in $\cO(\Sp^\Sigma(\CH))$ of maps
in $\cO J$ is a stable equivalence.
\end{lem}

\begin{proof}[Proof of Lemma~\ref{lem.operad.pout}]
 This follows as in~{\cite[8.7--8.10]{cm}}. 
Chadwick and Mandell consider pushouts of algebras over an operad $\cO$ for three
 different symmetric monoidal categories of spectra simultaneously
 (including $\Sp^\Sigma(\sset)$); all of their arguments hold as well for
 $\Sp^\Sigma(\CH)$ using the properties of h-cofibrations listed in
 Proposition~\ref{prop.h.list} and the generalization
 of~\cite[15.5]{mmss} given in Proposition~\ref{prop.gen.15.5} part
 (2).
\end{proof}

\begin{proof}[Proof of Theorem~\ref{thm.op.compare}]
This follows as in~{\cite[8.2]{cm}} 
again using
Proposition~\ref{prop.h.list} and Proposition~\ref{prop.gen.15.5} .
\end{proof}

\section{Commutative $H\Z$-algebras and $\spec(\sab)$}
\label{sec:hz}

In this section we consider the {Quillen equivalence} 
between $H\Z$-module
spectra and $\spec(\sab)$ and show that it also 
induces an equivalence
on the associated categories of commutative monoids.  Recall the
functor $Z$ from $H\Z$-modules to $\spec(\sab)$ from~\cite{s-alg}
which is given by $Z(M)
= \Zt(M) \sm_{\Zt H\Z} H\Z$ where $\Zt$ is the free abelian group on
the non-basepoint
simplices on each level. The right adjoint of $Z$ is given by
recognizing the unit in $\spec(\sab)$, $\Sym(\tilde{\Z}(\mathbb{S}^1))$,
as isomorphic to
$\Zt(\bS) \iso H\Z$. The right adjoint is labelled $U$ for underlying.
In~\cite[4.3]{s-alg}, the pair $(Z,U)$ was shown to induce a Quillen
equivalence on the standard model structures.  Since $Z$ is strong
symmetric monoidal, $(Z,U)$ also induces an adjunction between the
commutative monoids. We use the right induced model structure on
commutative monoids in  $\spec(\sab)$ and $H\Z$-module
spectra \cite[3.4.1]{ps2}.

\begin{thm}  \label{thm:hzvssab}
The functors $Z$ and $U$ induce a Quillen equivalence between
commutative $H\Z$-algebra spectra and commutative symmetric ring
spectra over $\sab$.

$$\xymatrix{Z\colon C(H\Z\text{-mod}) \ar@<0.5ex>[rr] &&
C(\spec(\sab)) :\! U  \ar@<0.5ex>[ll] }$$
\end{thm}

\begin{proof}
It follows from~\cite[proof of 4.3]{s-alg} that $U$ preserves and detects all
weak equivalences and fibrations since weak equivalences and
fibrations are determined on the underlying category of symmetric
spectra in pointed simplicial sets, $\spec(\sset_*)$. To
show that $(Z,U)$ is a Quillen equivalence, by~\cite[A.2 (iii)]{mmss}
it is enough to show that for all cofibrant
commutative $H\Z$ algebras $A$, the map $A \to UZ A$ is a stable
equivalence. If $A$ were in fact cofibrant as an $H\Z$ module
spectrum, this would follow from the Quillen equivalence on the module
level~\cite{s-alg}.  In the standard model structure on commutative
algebra spectra though, cofibrant objects are not necessarily
cofibrant as modules. The positive flat model (or $R$-model) structures
from~\cite[Theorem 3.2]{s-convenient} were developed for just this
reason. In Lemma~\ref{lem.flat.cofibrant} we show that for positive flat
cofibrant commutative $H\Z$ algebras $B$, the map $B \to UZB$ is a
stable equivalence.  It follows from Lemma~\ref{lem.flat.cofibrant}
that $A \to UZA$ is a stable equivalence for all standard (positive) cofibrant
commutative $H\Z$ algebras $A$, since such $A$ are also positive flat cofibrant
by~\cite[Proposition 3.5]{s-convenient}.  See also 
{\cite[8.10]{ps1}} for an alternative approach to this theorem.
\end{proof}

As discussed in the proof above, we next consider the flat model (or
$R$-model) structures from~\cite[Theorem 3.2]{s-convenient};  see
also~\cite[III, \S\S 2,3]{schwede-book}.

\begin{lem}\label{lem.flat.cofibrant}
For positive flat cofibrant commutative $H\Z$ algebras $B$, the map $B \to UZB$
is a stable equivalence.
\end{lem}

\begin{proof}
The crucial property for positive flat cofibrant ($H\Z$-cofibrant) commutative
monoids is that they are also (absolute) flat cofibrant as underlying modules.
Thus, if $B$ is a positive flat cofibrant commutative $H\Z$-algebra, then it is
also an (absolute) flat cofibrant $H\Z$-module by~\cite[Corollary
4.3]{s-convenient}.  (In fact $B$ is also a positive flat cofibrant $H\Z$-module by~\cite[Corollary
4.1]{s-convenient}, but we do not use that here.)  Since the Quillen equivalence
in~\cite[Proposition 4.3]{s-alg} is with respect to the standard model
structures~\cite[Proposition 2.9]{s-alg}, we next translate to that
setting.  Consider a cofibrant replacement $p\colon cB \to B$ in the
standard model structure on $H\Z$-modules; the map $p$ is a trivial
fibration and hence a level equivalence. Consider the commuting
diagram:

$$\xymatrix{cB\ar[d]_{}\ar [r]^(0.4){p}& B\ar [d]^{}\\
                      UZcB\ar [r]^{}&UZB}$$
\noindent
The left map is a stable equivalence by~\cite[Proposition 4.3]{s-alg}.
In Lemma~\ref{lem.flat.level} below we show that $Z$ takes level
equivalences between flat cofibrant objects to level equivalences.  By
~\cite[Proposition 2.8]{s-convenient}, $cB$ is flat cofibrant, so it
follows that the bottom map is also a stable equivalence.  Thus, we
conclude that the
{right} map is a stable equivalence as well.
\end{proof}

\begin{lem}\label{lem.flat.level}
The functor $Z$ takes level equivalences between flat cofibrant
objects to level equivalences.
\end{lem}

\begin{proof}
Here we will consider $Z$ as a composite of two functors and we will
always work over symmetric spectra in pointed simplicial sets, $\spec(\sset_*)$,
by forgetting from $\sab$ to $\sset_*$ wherever necessary.  The first
component is $\Zt$ from $H\Z$-modules to $\Zt H\Z$-modules, and the
second component is the extension of scalars functor $\mu_*$
associated to the ring homomorphism $\mu\colon \Zt H\Z \to H\Z$ induced by
recognizing $H\Z$ as isomorphic to $\Zt \bS$ and using the monad
structure on $\Zt$.

First, note that $\Zt$ is applied to each level and preserves level
equivalences as a functor from simplicial sets to simplicial abelian
groups.   The functor $\Zt$ also preserves flat cofibrations, and
hence flat cofibrant objects.  The generating flat cofibrations
($H\Z$-cofibrations) are of the form $H\Z \otimes M$ where $M$ is the
class of monomophisms of symmetric sequences.  Since $\Zt$ is strong
symmetric monoidal, these maps are taken to maps of the form $\Zt(H\Z)
\otimes \Zt(M)$.  Since $\Zt$ preserves monomorphisms, these are
contained in the generating flat ($\Zt H\Z$-) cofibrations, which are
of the form $\Zt H\Z \otimes M$.

Next, note that restriction of scalars, $\mu^*$,  preserves level
equivalences and level fibrations since they are determined as maps on
the underlying flat ($S$-)model structure; see the paragraph above
~\cite[Theorem 2.6]{s-convenient} and~\cite[Proposition
2.2]{s-convenient}.  It follows by adjunction that $\mu_*$ preserves
the flat cofibrations and level equivalences between flat cofibrant
objects.
\end{proof}

\begin{rem}
In the proof of Theorem \ref{thm:hzvssab} we use a reduction argument
that allows us to establish the desired Quillen equivalence by
checking that the unit map of the adjunction is a weak equivalence on
flat cofibrant objects in the flat model structure on commutative
$H\Z$-algebras. This approach avoids a discussion of a flat model
structure on commutative symmetric ring spectra in simplicial abelian
groups.
\end{rem}

\section{Dold-Kan correspondence for commutative monoids}

The classical Dold-Kan correspondence is an equivalence of categories
between the category of simplicial abelian groups, $\sab$, and the
category of non-negatively graded chain complexes of abelian groups,
$\chab$. In this section we establish a Quillen equivalence between
categories of commutative monoids in symmetric sequences of simplicial
abelian groups, $C(\sab^\Sigma)$,  and
non-negatively graded chain complexes, $C(\ch^\Sigma)$, carrying positive model
structures. {In the special case of pointed commutative monoids in
  symmetric sequences of simplicial modules and non-negatively graded
  chain complexes, such a Quillen equivalence is established in
  \cite[Theorem 6.5]{R}.}

In the next section we extend this equivalence from
symmetric sequences to symmetric spectra.  We first define the
relevant model structures on the
categories of symmetric sequences in simplicial abelian groups,
$\sigmas$, and chain complexes, $\sigmadg$.

\begin{defn} \label{def:modelstructure}
  \begin{itemize}
\item[]
  \item
Let $f\colon A \ra B$ be a morphism in $\sigmadg$. Then $f$ is a positive
weak equivalence, if $H_*(f)(\ell)$ is an isomorphism for positive levels
$\ell > 0$.
It is a positive fibration, if $f(\ell)$ is a fibration in the projective
model structure on non-negatively graded chain complexes for all $\ell >0$.
  \item
A morphism $g\colon C \ra D$ in $\sigmas$ is a positive fibration if $g(\ell)$
is a fibration of simplicial abelian groups in positive levels and it is a
positive weak equivalence if $g(\ell)$ is a weak equivalence for all
$\ell >0$.
  \end{itemize}
\end{defn}
In both cases, the positive cofibrations are determined by their left lifting
property with respect
to positive acyclic fibrations. Positive cofibrations are
cofibrations that are isomorphisms in level zero. One can check directly
that the above definitions give model category structures or use Hirschhorn's
criterion \cite[11.6.1]{Hi} and restrict to the diagram category whose
objects are natural
numbers greater than or equal to one and then use the trivial model
structure in level zero with
cofibrations being isomorphisms and weak equivalences and fibrations being
arbitrary. 
{The generating cofibrations are maps of the form
$G_r(i)$  for $r$ positive and such that  
$i$ is a generating cofibration in chain complexes
(simplicial modules). The generating acyclic cofibrations are maps of the form
$G_r(j)$ for $r$ positive and where 
$j$ is a generating acyclic cofibration in chain complexes
(simplicial modules).  }

We also get the corresponding
right-induced model structures on commutative monoids:

\begin{defn} \label{defn:positivemodelcomm}
An $f$  in $C(\sigmadg)(A,B)$ is a positive weak equivalence (fibration) if the
map on underlying symmetric sequences, $U(f)$ in $\sigmadg(U(A),U(B))$, is a
positive weak
equivalence (fibration). Similarly,  $g$ in $C(\sigmas)(C,D)$ is a positive
weak equivalence (fibration) if the map on underlying symmetric sequences,
$U(g) \in \sigmas(U(C),U(D))$ is a positive weak equivalence (fibration).
\end{defn}
In \cite[5.8, 6.2]{R} these model structures were established for
  \emph{pointed} commutative monoids in symmetric sequences of simplicial
  modules and non-negatively graded chain complexes. An object $A$ in
  $C(\sigmadg)$ or $C(\sigmas)$ is called pointed, if its zeroth level
  is the unit of the monoidal structure of the base category. We recall the key
points of the argument in the proof below. {This also makes it clear}
that the results of \cite{R} can be adapted to the setting of Definition
\ref{def:modelstructure}. 

\begin{lem} 
The structures defined in Definition \ref{defn:positivemodelcomm} yield
cofibrantly generated model categories where the generating cofibrations are
$C(G_r(i))$ and the generating acyclic cofibrations are $C(G_r(j))$ with $i,j$
as above and $r$ positive.
\end{lem}
\begin{proof}
Adjunction gives us that the maps with the right lifting property with respect
to all $C(G_r(j)), r>0$ are precisely the positive fibrations and the ones
with the
RLP with respect to all $C(G_r(i)), r>0$ are the positive acyclic fibrations.
Performing the small object argument based on the $C(G_r(j))$ for all positive
$r$  yields a factorization of any map as a positive acyclic cofibration and
a fibration
whereas the small object argument based on the $C(G_r(i))$ for positive $r$
gives the other factorization. 
\end{proof}

Let $\underline{\Z}$ denote the constant simplicial
abelian group with value $\Z$.
In the positive model structures cofibrant objects are commutative
monoids whose  zeroth level is isomorphic to $\underline{\Z}$ in $C(\sigmas)$
or to  $\Z[0]$ in $C(\sigmadg)$. {In particular, such objects
  are pointed in the sense of \cite[5.1]{R}.}

Let $\Gamma$ denote the functor from non-negatively graded chain
complexes to simplicial abelian groups that is the inverse of the
normalization functor. We can extend $\Gamma$ to a functor from
$\sigmadg$ to $\sigmas$ by applying $\Gamma$ in every level. As the
category of symmetric sequences of abelian groups is an abelian
category, the pair $(N, \Gamma)$ is still an equivalence of
categories.

{In the following we extend the result \cite[theorem 6.5]{R} in
the pointed setting, to the setting of positive model structures.} 
\begin{thm} \label{thm:dkpos}
Let  $\comsigmas$ and $ \comsigmadg$ carry the positive model structures. Then
the normalization functor $N\colon \comsigmas \ra \comsigmadg$ is the
right adjoint in a Quillen equivalence and its left adjoint is denoted $L_N$.
\end{thm}
\begin{proof}
A left adjoint $L_N$ to $N$ is constructed in \cite[6.4]{R}. As positive
fibrations
and weak equivalences are defined via the forgetful functors to $\sigmas$ and
$\sigmadg$, $N$ is a right Quillen functor and $N$ also detects weak
equivalences. Every object is fibrant, so we have to show that the unit of the
adjunction
$$ \eta\colon A \ra NL_N(A)$$
is a weak equivalence for all cofibrant $A \in \comsigmadg$. But cofibrant
objects are pointed and for these it is shown in \cite[proof of
theorem 6.5]{R} that the unit map is a weak equivalence.
\end{proof}


\section{Extension to commutative ring spectra} \label{sec:dk}
{
We will show that the pair $(L_N,N)$ gives rise to a Quillen
equivalence $(L_N,\phi^*N)$ on the level of commutative symmetric ring
spetra. }
\begin{lem} \label{lem:sym}
The Quillen pair $(L_N,N)$ satisfies
{$$ L_N(\mathrm{Sym} X_*) \cong \mathrm{Sym}(\Gamma(X_*))$$
for all non-negatively graded chain complexes $X_*$.}
\end{lem}
\begin{proof}
We can identify {$\mathrm{Sym}(C_*)$} 
with the free commutative monoid generated
by {$G_1X_*$, $C(G_1X_*)$.} 
Then, by definition of $L_N$ we obtain
{$$ L_N(C(G_1X_*)) \cong C(\Gamma(G_1X_*)) \cong C(G_1\Gamma(X_*)) \cong
\mathrm{Sym}(\Gamma(X_*)).$$}
\end{proof}

Let $\mathcal{C}$ be a category and let $A$ be an object of $\mathcal{C}$. Then
we denote by  $A \downarrow \mathcal{C}$ the category of objects under
$A$. 

\begin{cor} \label{cor:under1}
Let $ \comsigmadg$ and $\comsigmas$  carry the positive model category
structures 
and consider the induced model structures on the categories under a
specific object. Then the model categories
$\mathrm{Sym}(\Z[0]) \downarrow  \comsigmadg $ and
$ \mathrm{Sym}(\underline{\mathbb{Z}}) \downarrow \comsigmas$ are Quillen
equivalent.
\end{cor}
\begin{proof}
By Lemma \ref{lem:sym} we know that
$$L_N\mathrm{Sym}(\Z[0]) \cong \mathrm{Sym}(\underline{\mathbb{Z}}).$$
A direct calculation shows that
$N(\mathrm{Sym}(\underline{\mathbb{Z}}))$ is isomorphic 
to $\mathrm{Sym}(\Z[0])$. Therefore the Quillen equivalence
$(L_N,N)$ passes to a Quillen adjunction on the under categories. As the
classes of fibrations, weak equivalences and cofibrations in the under
categories are determined by the ones in the ambient category, this adjunction
is a Quillen equivalence.
\end{proof}
Note that {there is an isomorphism of categories between the
  category of} commutative monoids in $\Sp^\Sigma(\sab,
\tilde{\Z}(\mathbb{S}^1))$ {and the category}  
$\mathrm{Sym}(\tilde{\Z}(\mathbb{S}^1)) \downarrow
\comsigmas$. {A similar isomorphism of categories compares}  
commutative monoids in $\Sp^\Sigma(\chab, \Z[1])$ 
{and objects in} $\mathrm{Sym}(\Z[1]) \downarrow
\comsigmadg$. We can extend the 
Quillen equivalence from  \ref{cor:under1} to these under
categories. Recall from \cite[p.~358]{s-alg} that $\mathcal{N}$ is the
symmetric sequence in chain complexes with 
$N(\tilde{\Z}(\mathbb{S}^\ell))$ in level $\ell$. We denote by
$\mathbbm{1}$ the unit of the symmetric monoidal category
$\sigmadg$. This is the symmetric sequence with $\Z[0]$ in level zero
and zero in all positive levels.
\begin{prop} \label{prop:under2}
The functors {$(L_N,\Phi^*N)$} 
induce a Quillen equivalence on the model categories
$ \mathrm{Sym}(\Z[1]) \downarrow  \comsigmadg$ and
$ \mathrm{Sym}(\tilde{\Z}(\mathbb{S}^1))  \downarrow \comsigmas$ where
$ \comsigmadg$ 
and $\comsigmas$ carry the positive model structures. {Here, $\Phi^*$
is a suitable change-of-rings functor.} 
\end{prop}
\begin{proof}
As $\Gamma(\Z[1])$ is isomorphic to $\tilde{\Z}(\mathbb{S}^1)$ we obtain with
Lemma  \ref{lem:sym} that
$$L_N(\mathrm{Sym}(\Z[1]))\cong \mathrm{Sym}(\tilde{\Z}(\mathbb{S}^1)).$$

Therefore, if 
{$A$ is an object in $\mathrm{Sym}(\Z[1]) \downarrow
  \comsigmadg$,} 
then $L_N(A)$ 
is an object of $ \mathrm{Sym}(\tilde{\Z}(\mathbb{S}^1)) \downarrow \comsigmas$.
{We consider the functors}  
{
$$\xymatrix{
 {\mathrm{Sym}(\Z[1]) \downarrow  \comsigmadg} \ar[rr]^{L_N} & &
 {\mathrm{Sym}(\tilde{\Z}(\mathbb{S}^1))  \downarrow \comsigmas} \ar[dl]^N  \\
  & {\mathcal{N} \downarrow  \comsigmadg 
} \ar[ul]^{\Phi^*} &
 }$$}
where $\Phi\colon \mathrm{Sym}(\Z[1]) \ra \mathcal{N}$ is induced by the
shuffle transformation (see \cite[p.~358]{s-alg}) and $\Phi^*$ is the
associated change-of-rings map. Note that
$NL_N\mathrm{Sym}\Z[1] \cong \mathcal{N}$. Both functors $N$ and $\Phi^*$
preserve and detect {level and stable} weak equivalences
\cite[proof of 4.4]{s-alg}, {therefore they preserve and detect
  positive weak equivalences} and hence 
it suffices to show that
$$ A \ra \Phi^*NL_NA$$
is a weak equivalence in
$ \mathrm{Sym}(\Z[1]) \downarrow  \comsigmadg$ for all cofibrant objects
$\alpha\colon \mathrm{Sym}(\Z[1]) \ra A$.
There is a map of commutative monoids $\gamma \colon \mathbbm{1} \ra
\mathrm{Sym}(\Z[1])$
which is given by the identity in level zero and by the zero map in higher
levels. Let $\gamma^*$ be the associated change-of-rings
{functor: } 
$$\xymatrix{
{\mathbbm{1}} \ar[r]^(0.4)\gamma & {\mathrm{Sym}(\Z[1])} \ar[d]_\Phi
\ar[r]^\alpha &
{A} \ar[d]^{{\eta_A}} \\
 & {\mathcal{N}} \ar[r]^{NL_N\alpha} & {NL_NA}
}$$
{Note that $\eta_A \circ \alpha \circ \gamma = NL_N\alpha \circ
\Phi \circ \gamma$.} 
As $A$ is cofibrant, $\alpha(0)\colon \Z[0]=\mathrm{Sym}(\Z[1])(0) \ra A(0)$
is an isomorphism. Therefore $\gamma^*(A)$ is positively cofibrant as an object
in  $ \comsigmadg$. Hence we know that the map
$$ \gamma^*(A) \ra {\gamma^* \Phi^*NL_N(A)}
$$
is a positive weak equivalence in $ \comsigmadg$, \ie, a level equivalence in
all positive levels (it is also a weak equivalence in level
zero). {As $\gamma^*$ is the identity on objects} 
and only changes the module structure
we get that
$$ A \ra {\Phi}^*NL_N(A)$$
is a level equivalence in $\mathrm{Sym}(\Z[1])\downarrow  \comsigmadg$.
\end{proof}
\begin{rem}
With the positive model structure, $\comsigmadg$  is not left proper.
Consider for instance the map
$CG_r(0)= \mathbbm{1} \ra CG_r(\Z[0])$. This map is a cofibration for
positive $r$ 
in the positive model structure. On the other hand,
take the projection map from $\Z$ to $\Z/2\Z$. This yields a map $\pi$
in $\comsigmadg$ from the initial object $\mathbbm{1}$ to
$\mathbbm{1}/2\mathbbm{1}$ (where the latter object is
concentrated in level zero with value $\Z/2\Z[0]$).  As we
work in the positive model structure, this map is actually a weak equivalence.
If we push out $\pi$ along the cofibration $\mathbbm{1} \ra
CG_r(\Z[0])$ we get
$$g\colon  CG_r(\Z[0]) \ra  CG_r(\Z[0]) \odot \mathbbm{1}/2\mathbbm{1}. $$
In level $r$ this is the chain map
$$g(r)\colon G_r(\Z[0])(r) \cong \Z[\Sigma_r] \otimes \Z[0] \cong
\Z[\Sigma_r][0] \ra
\Z[\Sigma_r] \otimes \Z[0] \otimes \Z/2\Z[0] \cong \Z/2\Z[\Sigma_r][0].$$
Therefore we do not get an isomorphism for positive $r$ and the pushout of
the weak equivalence $\pi$  is not a weak equivalence.
\end{rem}

We want to transfer our results to a comparison of commutative monoids in
symmetric spectra of simplicial abelian groups and non-negatively graded chain
complexes where we consider the positive stable model structure.

\begin{lem} \label{lem:cofcof}
Cofibrant objects in $C(\Sp^\Sigma(\chab, \Z[1]))$ in the positive stable model
structure are cofibrant in  $\comsigmadg$.
\end{lem}
\begin{proof}
We can express the map $\mathbbm{1} \ra \mathrm{Sym}(\Z[1])$ as
$$ \mathbbm{1} \cong C(G_1(0)) \ra C(G_1(\Z[1])) = \mathrm{Sym}(\Z[1]).$$
Therefore the unit of $\mathrm{Sym}(\Z[1])$ is $C(G_1(i))$ with
$i\colon 0 \ra \Z[1]$ and hence it is a
cofibration and therefore the initial object $\mathrm{Sym}(\Z[1])$ of
$C(\Sp^\Sigma(\chab, \Z[1]))$
 is cofibrant in $\comsigmadg$

{As usual, let $\mathbb{S}^{n}$ denote the chain complex whose
  only non-trivial chain group is $\Z$ in degree $n$ and let
  $\mathbb{D}^n$ denote the chain complex with $\mathbb{D}^n_n =
  \mathbb{D}^n_{n-1}= \Z$ and $\mathbb{D}^n_i=0$ for all $i \neq n,
  n-1$ whose only non-trivial boundary map is the identity.} 
The cofibrant generators of the positive stable model structure are the maps

\begin{equation} \label{eq:generators}
\xymatrix@1{{\mathrm{Sym}(\Z[1]) \odot G_m(\mathbb{S}^{n-1})}
\ar[rrrr]^{\mathrm{Sym}(\Z[1]) \odot G_m(i_n)} & &&&
{\mathrm{Sym}(\Z[1]) \odot G_m(\mathbb{D}^{n})} }
\end{equation}
where $i_n$ is the cofibration of chain complexes $i_n \colon \mathbb{S}^{n-1}
\ra \mathbb{D}^{n}$ and $m \geq 1$. The $\odot$-product is the
coproduct in the category 
$\comsigmadg$ and therefore the map $\mathrm{Sym}(\Z[1]) \odot G_m(i_n)$ is the
coproduct of the identity map on $\mathrm{Sym}(\Z[1])$ and the map $G_m(i_n)$
and hence a cofibration in $\comsigmadg$.

Coproducts of generators as in \eqref{eq:generators} are cofibrations in
$\comsigmadg$ as well,
because the coproduct in $C(\Sp^\Sigma(\chab, \Z[1]))$ is given by the
$\odot_{\mathrm{Sym}(\Z[1])}$-product.

Every cofibrant object is a retract of a cell-object and these are sequential
colimits of pushout diagrams of the form
$$\xymatrix{
{X} \ar[d]_f \ar[r] & {A^{(n)}} \ar@{.>}[d] \\
{Y} \ar@{.>}[r]& {A^{(n+1)}}
}$$
where $f$ is a coproduct of maps like in \eqref{eq:generators} and $A^{(n)}$
is inductively constructed such that $A^{(0)}$ is $\mathrm{Sym}(\Z[1]))$. We
can inductively  assume that $X,Y$ and $A^{(n)}$ are cofibrant in
$\comsigmadg$. The pushout in $C(\Sp^\Sigma(\chab, \Z[1]))$ is the pushout in
$\comsigmadg$ and hence the pushout $A^{(n+1)}$ is cofibrant in  $\comsigmadg$
as well. Sequential colimits and retracts of cofibrant objects are cofibrant.
\end{proof}

\begin{thm} \label{thm:qecommspectrasabchab}
The Quillen pair $(L_N,{\Phi^*}N)$ induces a Quillen equivalence between
$C(\Sp^\Sigma(\chab, \Z[1]))$ and $C(\Sp^\Sigma(\sab, \tilde{\Z}(\mathbb{S}^1)))$
with the model structure{s} that {are} right-induced
from the positive 
stable model structures on the underlying categories of symmetric spectra.
\end{thm}
\begin{proof}
We have to show that the unit of the adjunction
$$ A \ra \Phi^*NL_NA$$
is a stable equivalence for all cofibrant $A \in C(\Sp^\Sigma(\chab, \Z[1]))$.
Lemma \ref{lem:cofcof} ensures that $A$ is cofibrant as an object in
$\comsigmadg$. Both $A$ and $\Phi^*NL_NA$ receive a unit map from
$\mathrm{Sym}(\Z[1]))$. As in the proof of Proposition \ref{prop:under2} we get
that
$$ \gamma^*A \ra NL_N\gamma^*A$$
is a level equivalence in $\comsigmadg$ and therefore the map $ A \ra \Phi^*NL_NA$ is
a level equivalence in $C(\Sp^\Sigma(\chab, \Z[1]))$ and hence a stable
equivalence.
\end{proof}

\section{Comparison of spectra in bounded and unbounded {chain complexes}}

\label{sec:chCH}

Recall that $\ch$ denotes the category of non-negatively graded chain
complexes and $\CH$ is the category of unbounded chain complexes of
abelian groups. There is a canonical inclusion functor
$i\colon \ch \ra \CH$ and a good truncation functor
$ C_0\colon \CH \ra \ch$
which assigns to an unbounded chain complex $X_*$ the non-negatively
graded chain complex $C_0(X_*)$ with
$$ C_0(X_*)_m = \begin{cases} X_m,  & m > 0,\\
\text{cycles}(X_0), & m = 0.  \end{cases} $$
We denote the induced functors on the corresponding categories of
symmetric spectra again by $i$ and $C_0$.
In this section we consider the Quillen equivalence
$$\xymatrix{i\colon\spec(\ch) \ar@<0.5ex>[rr] & & \spec(\CH) :\! C_0
  \ar@<0.5ex>[ll]  }$$
and show that it extends to a Quillen equivalence of categories of
commutative monoids.  The original Quillen equivalence is established
in~\cite[Proposition 4.9]{s-alg} for the usual stable model
structures.  Here we consider instead the positive stable model
structures from~\cite[\S 14]{mmss} and then consider the right induced model
structures on commutative monoids where $f$ is a weak equivalence or
fibration if it is an underlying positive weak equivalence or
fibration. \bltext{Note that the weak equivalences of
the stable model structure agree with the weak equivalences
of the positive stable model structure in $\spec(\ch, \Z[1])$ and $\spec(\CH, \Z[1])$. 
For this reason the positive and stable model structures are Quillen equivalent; see also~\cite[14.6]{mmss}.
It follows that the Quillen equivalence induced by $i$ and $C_0$ on the usual stable model structures
also induces a Quillen equivalence on the positive stable model structures.} 

\begin{prop}\label{prop.chain.adjoint}
The adjoint functors $i$  and $C_0$ form a Quillen equivalence between
the positive stable model structures on $\spec(\ch, \Z[1])$ and
$\spec(\CH, \Z[1])$.
\end{prop}

\begin{cor}\label{cor.adjunction}
Let $f$ be a \bltext{postive stably} fibrant replacement functor in \bltext{$\spec(\CH)$} 
and let $\eta\colon
X \ra C_0iX$ be the unit of the adjunction. The composite $X \to C_0iX
\to C_0 f i X$, is a stable equivalence for all objects $X$ in 
$\spec(\ch, \Z[1])$.
\end{cor}

\begin{proof}
It follows from the proof
of Proposition~\ref{prop.chain.adjoint} that the derived unit of the
adjunction is a weak
equivalence whenever $X$ is \bltext{positive} cofibrant.  Since \bltext{positive} trivial fibrations are
positive levelwise weak equivalences and a positive cofibrant replacement $cX
\to X$ is a \bltext{positive} trivial fibration, we only need to show that $C_0 f i$
preserves positive levelwise equivalences.  The inclusion $i$
preserves positive levelwise equivalences and $f$ preserves stable
equivalences. Any stable equivalence between positive stably fibrant
objects is a positive levelwise equivalence, so $f i$ preserves
positive levelwise equivalences.  Since $C_0$ preserves positive
levelwise equivalences between positive stably fibrant objects, the
corollary follows.
\end{proof}

\begin{cor} \label{cor:chCH}
The adjoint functors $i$  and $C_0$ induce a Quillen equivalence
between the commutative monoids in $\spec(\ch, \Z[1])$ and $\spec(\CH, \Z[1])$.
\end{cor}

\begin{proof}
Since the weak equivalences and fibrations are determined on the
underlying positive stable model structures, $C_0$ still preserves
fibrations and weak equivalences between positive stably fibrant
objects.  By~\cite[4.1.7]{HSS} it is then enough to check the derived
composite $C_0 i$ is a
stable equivalence for all cofibrant commutative monoids. This is
shown for all objects in Corollary~\ref{cor.adjunction}. The fibrant
replacement functor for commutative monoids will be different, but
the properties used in the proof of that corollary \bltext{still hold}, so we
conclude.
\end{proof}

\section{Quillen equivalence between $E_\infty$-monoids
in $\Chab$  and $\Sp^\Sigma(\Chab)$}

\label{sec:einfty}

We fix a cofibrant $E_\infty$-operad $\mathcal{O}$ in $\Chab$ {(in the
model structure on operads as in \cite[\S 2, Remark 2]{sp})} and we
consider the operad $F_0\mathcal{O}$ in symmetric spectra in chain
complexes. 

Let $\Chab$ carry the projective model structure and let
$E_\infty\Chab$ {denote the category of $\mathcal{O}$-algebras
  in $\Chab$ with its} right-induced model structure \cite[\S4, Theorem
4]{sp}. This model 
structure exists because $\Chab$ is a cofibrantly generated
monoidal model category, it satisfies the monoid axiom
\cite[3.4]{s-alg} {and $\cO$ is cofibrant.}  Alternatively, we
could work with Mandell's  
model structure on $E_\infty$-monoids in $\CH$ using the operad of the chains 
on the linear isometries operad \cite{mm-flat}. {See also
  \cite{bm1} for general existence results of model structures for
  categories of algebras over operads.}

Similarly, $\Sp^\Sigma(\Chab,\Z[1])$ with the stable model structure
is a cofibrantly generated monoidal model category satisfying the
monoid axiom \cite[3.4]{s-alg}, and as the set of generating acyclic
cofibrations
for the {positive stable} model structure on $\Sp^\Sigma(\Chab,\Z[1])$
is a subset of the ones for the stable structure, the positive stable
model category also satisfies the monoid axiom. We consider two model
structures for $E_\infty$-monoids in $\Sp^\Sigma(\Chab,\Z[1])$,
$E_\infty\Sp^\Sigma(\Chab,\Z[1])$:

\begin{itemize}
\item
We denote by $E_\infty\Sp^\Sigma(\Chab,\Z[1])^{s,+}$ the
model structure in which the  forgetful functor to the positive stable model
category structure on $\Sp^\Sigma(\Chab,\Z[1])$ determines the
fibrations and weak equivalences.
\item
Let
$E_\infty\Sp^\Sigma(\Chab,\Z[1])^{s}$ denote the model category
whose fibrations and weak equivalences are determined by the forgetful
functor to the stable model structure on
$\Sp^\Sigma(\Chab,\Z[1])$.
\end{itemize}

\begin{prop} \label{prop:positivevsordinary}
The model structure on $E_\infty\Sp^\Sigma(\Chab,\Z[1])^{s,+}$ is Quillen
equivalent to the model structure $E_\infty\Sp^\Sigma(\Chab,\Z[1])^s$.
\end{prop}
\begin{proof}
We consider the adjunction
$$ \xymatrix@1{{(E_\infty\Sp^\Sigma(\Chab,\Z[1])^s)}
\ar@<-0.5ex>[r]_{R} &   \ar@<-0.5ex>[l]_{L}
{(E_\infty\Sp^\Sigma(\Chab,\Z[1])^{s,+})} } $$
where $R$ and $L$ are both the identity functor. If $p$ is a fibration in the
stable model structure on $E_\infty\Sp^\Sigma(\Chab,\Z[1])$ then it is also
a positive stable fibration in $E_\infty\Sp^\Sigma(\Chab,\Z[1])$. Therefore $R$
preserves fibrations. As the weak equivalences in both model
structures agree,  $R$ is a right Quillen functor and it preserves and
reflects weak equivalences.  Hence the unit of the adjunction is a
weak equivalence.
\end{proof}

In the following we use
Hovey's comparison result \cite[9.1]{H}: Tensoring with $\Z[1]$
induces a Quillen autoequivalence on the category of unbounded chain
complexes, so we get that the pair $(F_0,\ev0)$ induces a Quillen
equivalence
$$ \xymatrix@1{ {\Chab} \ar@<0.5ex>[r]^(.3){F_0} & \Sp^\Sigma(\Chab,\Z[1])^s.
  \ar@<0.5ex>[l]^(.7){\ev0}
} $$
We can then transfer this Quillen equivalence to the corresponding
categories of $E_\infty$-monoids:
Both $F_0$ and $\ev0$ are strong symmetric monoidal functors. Fix a
cofibrant $E_\infty$-operad $\mathcal{O}$ in $\CH$ {as
  above}. As $\ev0 \circ 
F_0$ is the identity, $\ev0$ maps $F_0\mathcal{O}$-algebras in
$E_\infty\Sp^\Sigma(\Chab,\Z[1])$  to $\mathcal{O}$-algebras in
unbounded chain complexes.

\begin{thm} \label{thm:Qequeinftychains}
The functors $(F_0,\ev0)$ induce a  Quillen equivalence
$$\xymatrix{F_0\colon E_\infty\CH \ar@<0.5ex>[rr] & &
  \ar@<0.5ex>[ll] E_\infty\Sp^\Sigma(\Chab,\Z[1])^s :\!\ev0.}$$
\end{thm}
\begin{proof}
The proof follows Hovey's proof of \cite[5.1]{H}. It
is easy to see that $\ev0$ reflects weak equivalences between
stably fibrant objects: If $f\colon X \ra Y$ is such a map and $f(0)$ is a
weak equivalence, then $f(\ell)$ is a weak equivalence for all $\ell \geq 0$,
because $X$ and $Y$ are fibrant and $(-)\otimes \Z[1]$ is a
Quillen equivalence.

In our case $(-)\otimes \Z[1]$ is an equivalence of categories with inverse
the functor $\mathit{Hom}(\Z[1],-)$, where $\mathit{Hom}(-,-)$ is the internal
homomorphism bifunctor.

Therefore, for any $X$ in ${E_\infty}\Chab$, $F_0X$ is stably
fibrant because 
$$ (F_0X)_n = X\otimes \Z[n] \cong  \mathit{Hom}(\Z[1],X\otimes \Z[n+1])$$
and as every object in $\Chab$ is fibrant, $F_0X$ is always fibrant in
the projective model structure on $E_\infty\Sp^\Sigma(\Chab,\Z[1])$.

Let $A$ be a cofibrant object in  $E_\infty\Chab$. We have to show that
$$ \eta\colon A \ra \ev0 W (F_0A)$$
is a weak equivalence, for $W(-)$ the fibrant replacement in
$E_\infty\Sp^\Sigma(\Chab,\Z[1])$. But we saw that $F_0A$ is fibrant and
$A \ra \ev0 F_0A =A$ is the identity map, thus $\eta$ is a weak equivalence.
See also 
{\cite[8.10]{ps1}} for an alternative approach to this theorem.
\end{proof}

Observe that all of the Quillen equivalences that we have established so far
did not use any particular properties of $\Z$. We can therefore
generalize our results as follows.
\begin{cor} \label{cor:sum}
Let $R$ be a commutative ring with unit. There is a chain of
Quillen equivalences between the model category of commutative
$\HR$-algebra spectra and $E_\infty$-monoids in the category of
unbounded $R$-chain complexes.
\end{cor}
For $R=\Q$ we can strengthen the result:
\begin{cor} \label{cor:rat}
There is a chain of Quillen equivalences between the model category of
commutative $H\Q$-algebra spectra and differential graded commutative
$\Q$-algebras.
\end{cor}
\begin{proof}
It is well-known that the category of differential graded commutative
algebras and $E_\infty$-monoids in $\CH(\Q)$ possess a right-induced
model category structure and that there is a Quillen equivalence
between them. For a proof of these facts see {for instance} \cite[\S
7.1.4]{lurie-higheralg}. 
\end{proof}

\begin{rem}
Note that the proof of theorem \ref{thm:Qequeinftychains} applies in
broader generality: If $\cO$ is
an {arbitrary} operad in the category of chain complexes such
that right-induced model structure{s} on
$\cO$-algebras in $\CH$ and {on} $F_0(\cO)$-algebras in $\Sp^\Sigma(\CH,
\Z[1])^s$ exist, then the pair $(F_0,\ev0)$ yields a Quillen
equivalence between the model category of $\cO$-algebras in $\CH$ and
the model category of $F_0(\cO)$-algebras in $\Sp^\Sigma(\CH, \Z[1])^s$.
\end{rem}

\section{Symmetric spectra and $\mathcal{I}$-chain complexes}
\label{sec:ichains}
Let $\mathcal{I}$ denote the skeleton of the category of finite sets
and injective maps with objects the sets $\mathbf{n}=\{1,\ldots,n\}$ for $n
\geq 0$ with the convention that $\mathbf{0} = \varnothing$. The set of
morphisms $\mathcal{I}(\mathbf{p},\mathbf{n})$ consists of all
injective maps from
$\mathbf{p}$ to $\mathbf{n}$. In particular, this set is empty if $n$ is smaller
than $p$. The category $\mathcal{I}$ is a symmetric monoidal category
under disjoint union of sets.

For any category $\mathcal{C}$ we consider the diagram category
$\mathcal{C}^\mathcal{I}$ of functors from $\mathcal{I}$ to
$\mathcal{C}$. If $(\mathcal{C}, \otimes, e)$ is symmetric
monoidal, then
$\mathcal{C}^\mathcal{I}$ inherits a symmetric monoidal structure: For
$A, B \in \mathcal{C}^\mathcal{I}$ we set
$$ (A \boxtimes B)(\mathbf{n}) = \colim_{\mathbf{p}\sqcup \mathbf{q}
  \ra \mathbf{n}} A(\mathbf{p}) \otimes B(\mathbf{q}).$$
For details about $\mathcal{I}$-diagrams see \cite{sasch}. The
following fact is folklore; it was pointed out to the second author by
Jeff Smith in 2006 at the Mittag-Leffler Institute.

\begin{prop} \label{prop:ci}
Let $\mathcal{C}$ be any closed symmetric monoidal category with unit
$e$. Then the category $\Sp^\Sigma(\mathcal{C}, e)$
is {isomorphic} 
to the diagram category $\mathcal{C}^\mathcal{I}$.
\end{prop}
\begin{proof}
Let $X \in \Sp^\Sigma(\mathcal{C}, e)$. Then $X(n) \in
\mathcal{C}^{\Sigma_n}$ and
we have $\Sigma_n$-equivariant maps
$X(n) \cong X(n) \otimes e \ra  X(n+1)$, such that the composite
$$ \sigma_{n,p}\colon X(n) \cong X(n) \otimes e^{\otimes p} \ra X(n+1)
\otimes e^{\otimes p-1} \ra \cdots \ra X(n+p)$$
is $\Sigma_n \times \Sigma_p$-equivariant for all $n, p \geq 0$.

We send $X$ to $\phi(X) \in \mathcal{C}^\mathcal{I}$ with $\phi(X)(\mathbf{n})
= X(n)$. If $i=i_{p,n-p}\in \mathcal{I}(\mathbf{p},\mathbf{n})$ is the
standard inclusion, then we let
$\phi(i)\colon \phi(X)(\mathbf{p}) \ra \phi(X)(\mathbf{n})$ be
  $\sigma_{p,n-p}$. Every morphism $f \in
  \mathcal{I}(\mathbf{p},\mathbf{n})$ can be
written as $\xi \circ i$ where $i$ is the standard inclusion
and $\xi \in \Sigma_n$. For such $\xi$, the map $\phi(\xi)$ is given
by the $\Sigma_n$-action on $X(n) = \phi(X)(\mathbf{n})$.

{If $f = \xi' \circ i$ is another factorization of $f$ into the
standard inclusion followed by a permutation, then $\xi$ and $\xi'$
differ by a permutation $\tau \in \Sigma_n$ which maps all $j$ with $1
\leq j \leq p$ identically, \ie, $\tau$ is of the form $\tau =
\id_\mathbf{p} \oplus \tau'$ with $\tau' \in \Sigma_{n-p}$. As the
structure maps $\sigma_{p,n-p}$ are $\Sigma_p \times
\Sigma_{n-p}$-equivariant, the induced map $\phi(f) = \phi(\xi') \circ
\phi(i)$ agrees with $\phi(\xi) \circ
\phi(i)$.} 

The inverse of $\phi$, $\psi$, sends an $\mathcal{I}$-diagram in
$\mathcal{C}$, $A$,  to the symmetric spectrum $\psi(A)$ whose $n$th
level is $\psi(A)(n) = A(\mathbf{n})$. The $\Sigma_n$-action on $\psi(A)(n)$ is
given by the corresponding morphisms $\Sigma_n \subset
\mathcal{I}(\mathbf{n},\mathbf{n})$ and the structure maps of the spectrum are
defined as
$$\xymatrix@1{{\psi(A)(n) \otimes e^{\otimes p} =  A(\mathbf{n})
    \otimes e^{\otimes p}}
  \ar[r]^(.8)\cong & {A(\mathbf{n})} \ar[rr]^(.3){A(i_{n,p})} & &
  {A(\mathbf{n+p})=\psi(A)(n+p)}}.$$
The functors  $\phi$ and $\psi$ are well-defined and inverse to each
other.
\end{proof}
\begin{lem}
The functors $\phi$ and $\psi$ are strong symmetric monoidal.
\end{lem}
\begin{proof}
Consider two free objects $F_sC_*$ and $F_tD_*$ in
$\Sp^\Sigma(\mathcal{C}, e)$ for two chain complexes
$C_*$ and $D_*$. We know {in general \cite[\S 7]{H}} that
\begin{equation} \label{eq:freesmash}
F_sC_* \wedge F_tD_* \cong F_{s+t}(C_* \otimes D_*).
\end{equation}
Note that as an object in $\mathcal{C}^\mathcal{I}$ we have for
$\mathbf{n} \in \mathcal{I}$
$$ \phi(F_sC_*)(\mathbf{n}) = \Z\Sigma_n \otimes_{\Z\Sigma_{n-s}} C_*$$
for $n \geq s$ and zero otherwise. This coincides with the value of the free
$\mathcal{I}$-diagram on $\mathbf{n}$
$$ F_s^\mathcal{I}(C_*)(\mathbf{n}) = \Z\mathcal{I}(\mathbf{s},
\mathbf{n}) \otimes C_*$$
and in fact this yields an isomorphism of functors. Similarly,
$\psi(F_s^\mathcal{I}(C_*)) \cong F_sC_*$.

As the symmetric monoidal product in $\mathcal{C}^\mathcal{I}$ is given by
left Kan extension along the exterior product using the monoidal structure
of $\mathcal{C}$  we get
\begin{equation} \label{eq:freeboxtimes}
F_s^\mathcal{I}(C_*)  \boxtimes F_t^\mathcal{I}(D_*)\cong
F_{s+t}^\mathcal{I}(C_* \otimes D_*).
\end{equation}
From \eqref{eq:freesmash} we obtain that
$$ \psi(F_s^\mathcal{I}(C_*)) \wedge \psi(F_t^\mathcal{I}(D_*)) \cong
\psi(F_{s+t}^\mathcal{I}(C_*\otimes D_*)) \cong
\psi(F_s^\mathcal{I}(C_*) \boxtimes   F_t^\mathcal{I}(D_*))$$
and \eqref{eq:freeboxtimes} yields
$$ \phi(F_sC_*) \boxtimes \phi(F_tD_*) \cong \phi(F_{s+t}(C_*\otimes
D_*)) \cong \phi(F_sC_* \wedge F_tD_*).$$

The used isomorphisms are associative and compatible with the symmetry
isomorphisms. Every object in $\Sp^\Sigma(\mathcal{C}, e)$ and
$\mathcal{C}^\mathcal{I}$ can be written as a colimit of free objects and
as $\mathcal{C}$ is closed, the general case follows from the free case.
\end{proof}


\begin{rem}
{In \cite[3.3.9]{ps2} Pavlov and Scholbach describe explicitly
  (for a well-behaved symmetric monoidal model category $\mathcal{C}$) 
  how the unstable and stable model structures on
  $\Sp^\Sigma(\mathcal{C}, e)$ transfer to $\mathcal{C}^\mathcal{I}$
  under the above mentioned isomorphism of categories.} If $\mathcal{C}$ is
$\Chab$, their assumptions are satisfied. 

Note that the weak equivalences in $\Chab^\mathcal{I}$ have an
explicit description: they are the maps that become weak equivalences
after applying a {corrected homotopy 
colimit \cite[5.1]{dugger.simp}. This is the homotopy colimit of the
diagram where 
  every node is functorially replaced by a cofibrant object
  first.} To see this, consider Dugger's Bousfield localizations of 
diagram categories in \cite[\S 5]{dugger.simp}.
As the cofibrations and the fibrant objects in his model
structure in \cite[5.2]{dugger.simp} agree with ours, an argument due to Joyal,
\cite[E.1.10]{joyal}, ensures that we have the same class of weak
equivalences as well. 
\end{rem}

Taking a {cofibrant} $E_\infty$-operad $\mathcal{O}$ in $\Chab$ then
ensures that $\mathcal{O}$-algebras in $\Sp^\Sigma(\Chab,\Z[0])^s$ and
in  $\Chab^\mathcal{I}$ carry a model category structure such that
the forgetful functor determines fibrations and weak equivalences 

Since tensoring with the unit $\Z[0]$ is isomorphic to the identity,
  we can repeat  
all of the arguments in the previous section with $\Z[1]$ replaced by
$\Z[0]$. Thus  
we also obtain that the model category
$E_\infty\Sp^\Sigma(\Chab,\Z[0])^s$ is Quillen 
equivalent to the model category of $E_\infty$-monoids in $\Chab$.
Summarizing:

\begin{thm}
There is a chain of Quillen equivalences
$$ \xymatrix@1{ {E_\infty\Sp^\Sigma(\Chab,\Z[1])^s}
  \ar@<-0.5ex>[r]_(.7){\ev0} & {E_\infty\Chab} \ar@<-0.5ex>[l]_(.3){F_0}
  \ar@<0.5ex>[r]^(.3){F_0} &  {E_\infty\Sp^\Sigma(\Chab,\Z[0])^s}
  \ar@<0.5ex>[l]^(.7){\ev0}} $$
and the right-most model category is isomorphic to
$E_\infty\Chab^\mathcal{I}$.
\end{thm}

Last but not least we can connect commutative $\HR$-algebras to commutative
$\mathcal{I}$-chain {complexes}. {The positive stable model
structure on $\Sp^\Sigma(\CH(R),R[0])$ 
satisfies the assumptions of 
{\cite[5.10]{ps1}} 
and hence
commutative monoids and $E_\infty$
monoids in   $\Sp^\Sigma(\CH(R),R[0])^{s,+}$ carry model category
structures and there is a Quillen equivalence between them
\cite[3.4.1,3.4.4]{ps2}. This
yields that the model categories of commutative $\mathcal{I}$-chain
{complexes}, 
$C(\CH(R)^{\mathcal{I},+})$,
and $E_\infty$ $\mathcal{I}$-chain {complexes},
$E_\infty(\CH(R)^{\mathcal{I},+})$ 
are Quillen equivalent, if we take the model structure that is
right induced from the positive model structure on
$\CH(R)^{\mathcal{I},+}$. } 
\begin{thm}\label{comm.hz.I.ch}
There is a chain of Quillen equivalences between the model categories of
commutative $\HR$-algebra spectra, $C(\HR\text{-mod})$, and
commutative monoids in
the category $\CH(R)^\mathcal{I}$ where the latter carries the
right-induced model structure
from the positive model structure on $\CH(R)^\mathcal{I}$,
$\CH(R)^{\mathcal{I},+}$.
\end{thm}

We close with an {important} example of a commutative
$\mathcal{I}$-chain complex.   Consider a chain complex $C_*$ together with a
$0$-cycle, \ie,  
with a map $\eta\colon \Z[0] \ra C_*$. The assignment 
$ \mathbf{n} \mapsto C_*^{\otimes n}$ 
defines a functor $\mathsf{sym}$ from $\mathcal{I}$ to the category of
unbounded chain complexes {(namely $\Sym(C_*)$). 
Schlichtkrull shows in  \cite{schl} that $\mathsf{sym}$  
is the algebraic analogue of the symmetric product in the category of spaces. }

\end{document}